\documentclass[12pt]{amsart}

\usepackage{amscd,amsthm,amsfonts,amssymb,amsmath}
\usepackage{mathrsfs}
\usepackage[colorlinks=true]{hyperref}
\usepackage{fullpage}
\usepackage[all]{xy}

    \newcommand{\BA}{{\mathbb {A}}} \newcommand{\BB}{{\mathbb {B}}}
    \newcommand{\BC}{{\mathbb {C}}} 
     \newcommand{\BF}{{\mathbb {F}}}

     \newcommand{\BP}{{\mathbb {P}}}
    \newcommand{\BQ}{{\mathbb {Q}}} \newcommand{\BR}{{\mathbb {R}}}

     \newcommand{\BZ}{{\mathbb {Z}}}

     \newcommand{\CH}{{\mathcal {H}}}

    \newcommand{\CO}{{\mathcal {O}}} 
     \newcommand{\CR}{{\mathcal {R}}}
    \newcommand{\CS}{{\mathcal {S}}}

     \newcommand{\fl}{{\mathfrak{l}}}

     \newcommand{\fP}{{\mathfrak{P}}}

    \newcommand{\ab}{{\mathrm{ab}}}

    \newcommand{\Aut}{{\mathrm{Aut}}}

    \newcommand{\End}{{\mathrm{End}}}

    \newcommand{\Frob}{{\mathrm{Frob}}}

    \newcommand{\Gal}{{\mathrm{Gal}}} \newcommand{\GL}{{\mathrm{GL}}}
    
    \newcommand{\Hom}{{\mathrm{Hom}}}

    \newcommand{\ord}{{\mathrm{ord}}} \newcommand{\rk}{{\mathrm{rank}}}
     \newcommand{\Pic}{\mathrm{Pic}}
    \newcommand{\pr}{{\mathrm{pr}}} 
    
    \renewcommand{\mod}{\ \mathrm{mod}\ }

    \newcommand{\Sel}{{\mathrm{Sel}}}

    \newcommand{\SL}{{\mathrm{SL}}}

    \newcommand{\tor}{{\mathrm{tor}}}

    \newcommand{\Vol}{{\mathrm{Vol}}}

     \newcommand{\M}{\mathrm{M}}
        \newcommand{\Tr}{\mathrm{Tr}}

      \newcommand{\sK}{\mathscr{K}}

\newcommand{\matrixx}[4]{\begin{pmatrix}
#1 & #2 \\ #3 & #4
\end{pmatrix} }        

    \font\cyr=wncyr10

    \newcommand{\Sha}{\hbox{\cyr X}}\newcommand{\wt}{\widetilde}
    \newcommand{\wh}{\widehat}

    \newcommand{\ov}{\overline}

    \newcommand{\lra}{\longrightarrow}
    \newcommand{\ra}{\rightarrow}

    \newcommand{\N}{\mathrm{N}}

                              \newcommand{\hilbert}[4]{\left(\frac{#1,#2}{#3;#4}\right)}

    \theoremstyle{plain}
    \newtheorem{thm}{Theorem}[section] \newtheorem{coro}[thm]{Corollary}
    \newtheorem{lem}[thm]{Lemma}  \newtheorem{prop}[thm]{Proposition}

\theoremstyle{remark} 
\theoremstyle{remark} 
\theoremstyle{remark} 

    \numberwithin{equation}{section}

\begin{document}
\title{Cube sums of form $3p$ and $3p^2$}
\author{Jie Shu, Xu Song and Hongbo Yin}
\begin{abstract}
Let  $p\equiv 2,5\mod 9$ be an odd prime. In this paper, we prove that at least one of $3p$ and $3p^2$ is a cube sum by constructing certain nontrivial Heegner points. We also establish the explicit Gross-Zagier formulae for these Heegner points  and give variants of the Birch and Swinnerton-Dyer conjecture of the related elliptic curves.  
\end{abstract}
\address{School of Mathematical Sciences, Tongji University, Shanghai 200092}
\email{shujie@tongji.edu.cn}
\address{Department of Mathematics, National University of Singapore, Singapore 119076}
\email{xusong@u.nus.edu}
\address{Academy of Mathematics and Systems Science, Morningside center of Mathematics, Chinese Academy of Sciences, Beijing 100190}
\email{yinhb@math.ac.cn}
\maketitle
\tableofcontents

\section{Introduction}
A nonzero rational number is called a cube sum if it is of the form $a^3+b^3$ with $a,b\in \BQ^\times$. 
For any $n\in \BQ^\times$, let $E_n$ be the elliptic curve over $\BQ$ defined by the projective equation $$x^3+y^3=nz^3.$$ If $n=a^3$ is a cube of a rational number $a\in \BQ^\times$, then $E_n(\BQ)\simeq \BZ/3\BZ$ is generated by the point with projective coordinates $(a:0:1)$, and then $n$ is not a cube sum. If $n=2a^3$ is twice a cube of a rational number, then $E_n(\BQ)\simeq \BZ/2\BZ$ has a generator $(a:a:1)$, and $n$ is a cube sum. Otherwise, the torsion subgroup $E_n(\BQ)_\tor$ is trivial. Therefore, if  $n$ is not twice a cube of rationals , then $n$ is a cube sum if and only if $\rk_\BZ E_n(\BQ)>0$. Determining which rational numbers are cube sums is a classical Diophantine problem, which may also be referred as the Sylvester problem for historical reasons \cite{Sylv79a,Sylv79b,Sylv80a,Sylv80b}. 

Various investigations of the Sylvester problems can be found in the literature. In \cite{Satge} (see also \cite{DV1,CST17}), P. Satg\'e proved that if $p\equiv 2$ resp. $5 \mod 9$, then $2p$ resp. $2p^2$ is a cube sum. D. Lieman \cite{Lieman} showed there are infinitely many cube-free integers, which are not cube sums. In \cite{CST17}, Cai-Shu-Tian proved there are infinitely many integers, with given even number of distinct prime factors,  which are cube sums (resp. not cube sums).
The famous Sylvester conjecture states that any prime number congruent to $4,7,8\mod 9$ is a cube sum.  A partial result can be found in the recent work of Dasgupta and Voight \cite{DV17}. 

In this paper, we are concerned with cube sums of the form $3p^i$ where $p$ is a prime and $i=1,2$, and the main result is as follows.
\begin{thm}\label{main}
Let  $p\equiv 2,5\mod 9$ be an odd prime. Then at least one of $3p$ and $3p^2$ is a cube sum.
\end{thm}
Let $p\equiv 2,5\mod 9$ be an odd prime number. By the descent method (Proposition \ref{selmer}) and  the formula for epsilon factors in \cite{Liverance}, we know
\[\Sel_3(E_{3p})\simeq \Sel_3(E_{3p^2})\simeq \BZ/3\BZ,\]
and
\[\epsilon(E_{3p})=\epsilon(E_{3p^2})=-1.\]
Then the Birch and Swinnerton-Dyer (B-SD) conjectures for the elliptic curves $E_{3p}$ and $E_{3p^2}$ imply that 
\[\rk E_{3p}(\BQ)=\rk E_{3p^2}(\BQ)=1.\]
Therefore, the B-SD conjecture implies that both $3p$ and $3p^2$ should be cube sums.

Let $K=\BQ(\sqrt{-3})\subset \BC$ be an imaginary quadratic field with $\CO_K=\BZ[\omega]$ its ring of integers, $\omega=\frac{-1+\sqrt{-3}}{2}$. For any $n\in \BQ^\times$, the elliptic curve $E_n$ has Weierstrass equation $y^2=x^3-2^4\cdot 3^3\cdot n^2$, and has complex multiplication by $\CO_K$ over $K$. We fix the complex multiplication $[\ ]:\CO_K\simeq \End_{\ov{\BQ}}(E_n)$ by $[\omega](x,y)=(\omega x,y)$. 

Let $p\equiv 2,5\mod 9$ be an odd prime.  Let $\chi_1:G_K\ra \CO_K^\times$ resp. $\chi_2$ be the character given by $\chi_1(\sigma)=(\sqrt[3]{3p})^{\sigma-1}$ resp. $\chi_2(\sigma)=(\sqrt[3]{3p^2})^{\sigma-1}$. 
We normally embed $K$ into $\M_2(\BQ)$ with a fixed point $\tau=p\omega /9\in \CH$ under linear fractional transformations, where $\CH$ is the Poinc\'are upper half plane.  The elliptic curve $E_1$ is naturally isomorphic to the modular curve $X_0(27)$, and let $f:X_0(27)\ra E_1$ be the isomorphic modular parametrization.  Let $P_\tau\in X_0(27)$ be the point arising from the imaginary quadratic $\tau\in \CH$. The CM point $P_\tau$ is defined over the ring class field $H_{9p}$.  Define the Heegner points 
\[ R_1=\Tr_{H_{9p}/L_{(3p)}} f(P_\tau)\in E_1(L_{(3p)}),\]
 and
 \[R_2=\Tr_{H_{9p}/L_{(3p^2)}} f(P_\tau)\in E_1(L_{(3p^2)}).\]
By Proposition \ref{decomposition}, these Heegner points give rise to rational points in $E_{3p}(K)$ and $E_{3p^2}(K)$ respectively, and hence  if the  Heegner point $R_1$ resp. $R_2$ is of infinite order, then $3p$ resp. $3p^2$ is a cube sum.

The base change L-function $L(s,E_1,\chi_1)$ resp. $L(s,E_1,\chi_2)$ has a decomposition
\[L(s,E_1,\chi_1)=L(s,E_{3p})\cdot L(s,E_{9p^2}) \textrm{ resp. }L(s,E_1,\chi_2)=L(s,E_{3p^2})\cdot L(s,E_{9p}).\]
By the formulae for epsilon factors in \cite{Liverance}, the L-function $L(s,E_1,\chi_1)$ resp. $L(s,E_1,\chi_2)$ has sign $-1$ in the functional equation.
The morphism $f$  is a test vector for  the pair $(E_1,\chi_1)$ resp. $(E_1,\chi_2)$ in the sense of Proposition \ref{test-vector}, and we have the following explicit height formula (see Theorem \ref{thm:GZ}). 
\begin{thm} \label{thm:GZ}
Let  $p\equiv 2,5 \mod 9$ be an odd prime number.  We  have the following explicit height formula of Heegner points:
\[\frac{L'(1,E_{3p})L(1,E_{9p^2})}{\Omega_{3p}\Omega_{9p^2}}=\wh{h}_\BQ(R_1),\]
and
\[\frac{L'(1,E_{3p^2})L(1,E_{9p})}{\Omega_{3p^2}\Omega_{9p}}=  \wh{h}_\BQ(R_2).\]
where $\Omega_n$ denotes the minimal real period of the elliptic curve $E_n$ and $\wh{h}_\BQ$ denotes the N\'eron-Tate height on $E_1$ over $\BQ$.

\end{thm}

Define the Heegner point
\[z=\Tr_{H_{9p}/L_{(3,p)}}f(P_\tau)\in E_1(L_{(3,p)}).\]
It follows from Proposition \ref{decomposition} that, modulo torsion points, we have
\[R_1+R_2=3z.\]
Then the main Theorem \ref{main} follows from the nontriviality of the Heegner point $z$. By a refinement of the Kronecker congruence for modular functions at supersingular points (Proposition \ref{prop: Kronecker}), we can analyze the behavior of the Heegner point $z$ modulo $p$. And the nontriviality of the Heegner point $z$ follows from the fact that the reduction of $z\mod p$ always has a nontrivial $3$-component, which is achieved by Proposition \ref{finite} and Theorem \ref{thm:1}. 

In the final subsection, by comparing the explicit height formulae with the B-SD conjectures, we give the variants of the B-SD conjectures of the relate elliptic curves in terms of Heegner points.

For any integer $c\geq 1$, let $\CO_c$ be the order of $K$ of conductor $c$ and let $H_c$ be the ring class field of conductor $c$ which is characterized by the isomorphism
\[\Gal(H_c/K)\simeq \Pic(\CO_c)\]
under class field theory. We will using the following field-extension diagram:
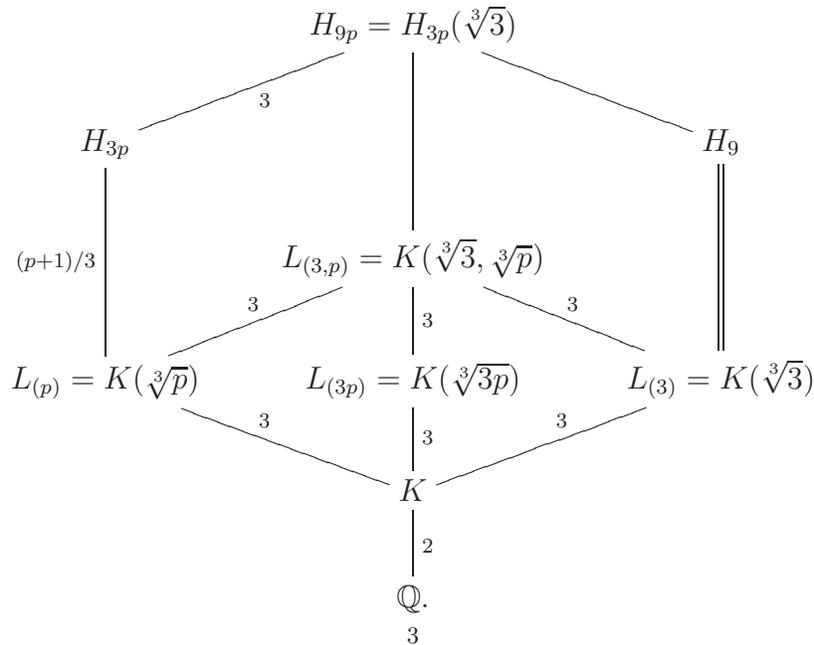
\[\xymatrix{&H_{9p}=H_{3p}(\sqrt[3]{3})\ar@{-}[dl]^{3}\ar@{-}[dr]\ar@{-}[dd]&\\
            H_{3p}\ar@{-}[dd]_{(p+1)/3}&&H_9\ar@{=}[dd]\\
            &L_{(3,p)}=K(\sqrt[3]{3},\sqrt[3]{p})\ar@{-}[dr]^{3}\ar@{-}[d]^{3}\ar@{-}[dl]_{3}&\\
            L_{(p)}=K(\sqrt[3]{p})\ar@{-}[dr]^{3}&L_{(3p)}=K(\sqrt[3]{3p})\ar@{-}[d]^{3}&L_{(3)}=K(\sqrt[3]{3})\ar@{-}[dl]_{3}\\
            &K\ar@{-}[d]^{2}&\\
            &\BQ.&\\
            }\]

\noindent {\bf Acknowledgements.} The authors would like to thank professor Ye Tian his for long-term support and  encouragement. The second-named author (X. Song) would like to thank professor Chen-Bo Zhu for his encouragement and support. We would like to thank John Voight for his suggestion to use the Kronecker congruence to compute the coordinates modulo $p$ in the supersingular case. We also would like to thank Jinbang Yang for helpful discussions which lead a proof of Proposition \ref{finite}, and thanks are also due to Sheng Meng for valuable discussions on algebraic geometry.

\section{Modular Curves}
\subsection{The Modular Curve $X_0(27)$}
The group $\GL_2(\BQ)^+$ of invertible matrices with rational entries and positive determinants acts on the Poinc\'are upper half plane $\CH$ by linear fractional transformations:
\[\matrixx{a}{b}{c}{d}z=\frac{az+b}{cz+d},\quad z\in \CH.\]Let $\Gamma_0(27)$ be the subgroup of $\SL_2(\BZ)$ consisting of all matrices
$$\left (\begin{array}{cc}a&b\\c&d
\end{array} \right ) \textrm{ with $c\equiv 0\mod 27$,}$$
and define \[Y_0(27)=\Gamma_0(27)\backslash \CH \textrm{ and } C=\Gamma_0(27)\backslash  \BP^1(\BQ).\]
Then $Y_0(27)$ is an affine smooth curve defined over $\BQ$ and the modular curve $X_0(27)$ of level $\Gamma_0(27)$ is defined to be its projective closure. The underlying Riemann surface of $X_0(27)$ is given as
\[X_0(27)(\BC)=Y_0(27)(\BC)\bigsqcup C.\]
For each $z\in \CH\sqcup \BP^1(\BQ)$, denote by $[z]$ the point on $X_0(27)$ represented by $z$.
Elements in
\[ C=\{[0],[\pm 1/3],[\pm 1/9],[\infty]\}\]
are called cusps of $X_0(27)$ and the cusps $[0],[\infty]$ are rational and the other four cusps are defined over the imaginary quadratic field $K=\BQ(\sqrt{-3})$. The modular curve $X_0(27)$ is of genus $1$ with a rational cusp $[\infty]$, so we have an elliptic curve $(X_0(27),[\infty])$ over $\BQ$ with $[\infty]$ as its zero element.

Define $N$ to be the normalizer of $\Gamma_0(27)$ in $\GL_2(\BQ)^+$. The linear fractional action of $N$ on $X_0(27)$ induces an injective homomorphism
\[\Phi:N/\BQ^\times\Gamma_0(27)\hookrightarrow \Aut_\BC(X_0(27))=\CO_K^\times\ltimes X_0(27)(\BC).\]
Put
\[W=\matrixx{0}{1}{-27}{0},\quad A=\matrixx{1}{1/3}{0}{1},\]
and
\begin{equation}\label{equ:1}
B=-\frac{1}{27}WA^{-1}W=\matrixx{1}{0}{9}{1},\quad C=-\frac{1}{27}WB^{-1}WB=\matrixx{4}{1/3}{9}{1}.
\end{equation}
It follows from \cite[Theorem 8]{AL1970} and \cite{Ogg80} that the quotient group $N/\BQ^\times\Gamma_0(27)$ is generated by $W$ and $A$, and moreover, we have
\[N/\BQ^\times\Gamma_0(27)= \langle B\rangle^{\BZ/3\BZ} \ltimes \left(\langle W\rangle^{\BZ/2\BZ}\ltimes\langle C\rangle^{\BZ/3\BZ}\right).\]

\begin{prop}\label{mc}
\begin{itemize}
\item[1.] The elliptic curve $(X_0(27,[\infty]))$ is isomorphic to $E_1$ and we identify $X_0(27)$ with $E_1$ so that the cusp $[0]$ has coordinates $(12,36)$. Moreover, the isomorphism can be given in terms of quotients of the Dedekind $\eta$-functions:
\begin{eqnarray*}
&X_0(27)\cong E_1:y^2=x^3-432,\\
&z\in \CH\mapsto \left(4\frac{\eta(9z)^4}{\eta(3z)\eta(27z)^3},8\frac{\eta(3z)^3}{\eta(27z)^3}+36\right).
\end{eqnarray*}
\item[2.] We have an embedding
\[\Phi: N/\BQ^\times\Gamma_0(27)\hookrightarrow \CO_K^\times\ltimes C.\]
For any point $P\in X_0(27)$, we have
\[\Phi(A)(P)=[\omega]P,\quad \Phi(W)(P)=[-1]P+[0],\]
and
\[\Phi(B)(P)=[\omega^2]P+[1/9],\quad \Phi(C)(P)=P+[1/9].\]
\end{itemize}
\end{prop}
\begin{proof}
For the modular isomorphism of $X_0(27)$ with the elliptic curve $E_1$ in terms of quotients of the Dedekind eta-functions, we refer to \cite[2.5]{DV1} and note that we changed the Weierstrass equation.

The second assertion follows from a similar argument as in \cite[Proposition 2.2]{CST17} which deals with the modular curve $X_0(36)$. Since the modular automorphism $\Phi(A)$ is of order $3$ and fixes the cusp $[\infty]$, as an automorphism on the elliptic curve $E_1$, $\Phi(A)$ must be the complex multiplication by a third root of unity.   Let $\phi=\sum_{n\geq 1}a_nq^n$ be the unique newform of level $\Gamma_0(27)$ and weight $2$. Then the N\'eron differential on $X_0(27)$ is $dx/2y=\phi(q)dq/q$. At $[\infty]$, it is represented by $dq$, since $\phi$ is a normalised cusp form. Then $\Phi(A)^*(dq)=\omega dq$. On the other hand, $[\omega](x,y)=(\omega x,y)$ and then $[\omega]^*(dx/2y)=\omega dx/2y$. So we have $\Phi(A)=[\omega]$. 

Since $\Phi(W)$ is of order $2$, it must be of the form 
\[\Phi(W)(P)=[\pm 1] P+T\]
where $T$ is a nontrivial torsion point. If $\Phi(W)(P)=P+T$, then $T$ is a point of exact order $2$ and $\Phi(W)$ has no fixed point which contradicts with the fact that $\Phi(W)$ has a fixed point  represented by $\sqrt{\frac{-1}{27}}\in \CH$. So \[\Phi(W)(P)=[-1] P+T\]
and if we take $P=[\infty]$, we see immediately that $T=[0]$. 

Then we see that for any point $P\in X_0(27)$, we have
\[\Phi(A)(P)=[\omega]P,\quad \Phi(W)(P)=[-1]P+[0],\]
and the formulae for $\Phi(B)$ and $\Phi(C)$ follow directly from (\ref{equ:1}).
\end{proof}
\begin{coro}
The coordinates of the cusps under the Weierstrass equation $$E_1:y^2=x^3-432$$ are given as follows:
\[\begin{aligned}
&[\infty]=O,&&[0]=(12,36),&&[1/3]=(12\omega,36),\\
&[-1/3]=(12\omega^2,36),&&[1/9]=(0,-12\sqrt{-3}),&&[-1/9]=(0,12\sqrt{-3}).
\end{aligned}\]
\end{coro}
\begin{proof}
Put 
\[X(z)=\frac{\eta(9z)^4}{\eta(3z)\eta(27z)^3},\quad Y(z)=\frac{\eta(3z)^3}{\eta(27z)^3}.\]
Recall the Dedekind $\eta$-function
\[\eta(z)=q^{1/24}\prod_{n\geq 1}(1-q^n),\quad q=e^{2\pi i z},\]
and we have the transformation formulae:
\[\eta(z+1)=e^{2\pi i/24}\eta(z),\quad \eta(-1/z)=\sqrt{\frac{z}{i}}\eta(z).\]
Then
\[X\left(\frac{-1}{27z}\right)=\frac{\eta\left(\frac{-1}{3z}\right)^4}{\eta\left(\frac{-1}{9z}\right)\eta\left(\frac{-1}{z}\right)^3}=3\frac{\eta(3z)^4}{\eta(9z)\eta(z)^3}=3\prod_{n\geq 1}\frac{1-q^{3n}}{(1-q^{9n})(1-q^{n})^3},\]
\[Y\left(\frac{-1}{27z}\right)=\frac{\eta\left(\frac{-1}{9z}\right)^3}{\eta\left(\frac{-1}{z}\right)^3}=27\frac{\eta(9z)^3}{\eta(z)^3}=27q\prod_{n\geq 1}\frac{(1-q^{9n})^3}{(1-q^n)^3}.\]
Taking $q\ra +\infty$ as $z\ra 0i$, we see
\[[\infty]=O.\]
Taking $q\ra 0$ as $z\ra +\infty i$, we see
\[[0]=(12,36).\]
Similarly, we have
\[X(z+1/3)=\frac{\eta(9z+3)^4}{\eta(3z+1)\eta(27z+9)^3}=e^{2\pi i/3}X(z),\]
\[Y(z+1/3)=\frac{\eta(3z+1)^3}{\eta(27z+9)^3}=Y(z).\]
Taking $z\ra 0$, we see
\[[1/3]=(12\omega,36).\]

By the action of the linear fractional transformation $A$ and the formula in Proposition \ref{mc}, we have 
\[[-1/3]=\Phi(A)([1/3])=[\omega][1/3]=(12\omega^2,36).\]

By the action of the Atkin-Lehner operator $W$ and the formula in Proposition \ref{mc}, we have 
\[[-1/9]=\Phi(W)([1/3])=[-1][1/3]+[0]=(0,12\sqrt{-3}),\]
and similarly, we can compute
\[[1/9]=(0,-12\sqrt{-3}).\]

\end{proof}

It is easy to see that
\[E_1(K)=E[3]=\{O,(12 \omega ^i,\pm 36),(0,\pm 12\sqrt{-3})\}_{i=0,1,2}.\]

\begin{prop}\label{torsion}
We have $E_1(L_{(3,p)})_\tor=E_1(K)$.
\end{prop}
\begin{proof}
Let $\ell$ be a prime. Suppose $P\in E_1[\ell]$ is a torsion point of order $\ell$. By \cite[Theorem 5.1]{Shimurabook}, the field $K(P)$ contains a subfield over $K$ with Galois group $(\CO_K/\fl)^\times$, where $\fl$ is a prime ideal of $K$ above $\ell$. If $\ell\neq 2,3,7$,  then $[K(P):K]>9$. If $\ell=7$, then $(\CO_K/\pi)^\times=6$ and $6\mid [K(P):K]$. 


For $\ell=2$, we see that $K(E_1[2])/K$ has degree $3$,  and is unramified at $p$. Thus $$K(E_1[2])\cap L_{(3,p)}\subset L_{(3)}.$$ Since 
\begin{equation}\label{torsion}E_1(L_{(3)})_\tor \simeq \BZ/3\BZ\times\BZ/3\BZ,
\end{equation} 
we conclude that $E_1(L_{(3,p)})[2]=0$.

Finally, suppose $P\in E_1(L_{(3,p)})[3^\infty]$. Since $E_1$ has good reductions outside $3$, we conclude that $K(P)$ is unramified outside $3$, which forces that $K(P)\subset L_{(3)}$. Again by the fact (\ref{torsion}), we see that $$E_1(L_{(3,p)})[3^\infty]=E_1[3]=E_1(K).$$
\end{proof}

\subsection{Base Change to $K=\BQ(\sqrt{-3})$}
Let $U_0(27)\subset\GL_2(\wh{\BZ})$ be the subgroup consisting of matrices
$\left (\begin{array}{cc}a&b\\c&d
\end{array} \right )$ with $c\equiv 0\mod 27$, and let $U\subset U_0(27)$  be the subgroup consisting of matrices with $a\equiv d\mod 3$. Let $X_U$ be the modular curve over $\BQ$ whose underlying Riemann
surface is
\[X_U(\BC)=\GL_2(\BQ)^+\left\backslash \left(\left(\CH\bigsqcup\BP^1(\BQ)\right )\times \GL_2(\BA_f)\right/U\right).\]
Indeed, $X_0(27)$ is the modular curve $X_{U_0(27)}$ and therefore $X_U$ is a natural double cover of $X_0(27)$. Under class field theory, $ \BQ^\times_+\wh{\BZ}^\times/\BQ^\times_+\det(U)\simeq \Gal(K/\BQ).$
Noting that $\GL_2(\BQ)^+\cap U=\Gamma_0(27)$, we see that the
modular curve $X_U$ is isomorphic to $X_0(27)\times_\BQ K$ as a curve over $\BQ$ (cf. \cite[Chapter 6]{Shimurabook}). In adelic language, if $z\in \CH\bigsqcup \BP^1(\BQ)$ and $g\in \GL_2(\BA_f)$, we denote $[z,g]_U\in X_U$ the point corresponding to the coset of the pair $(z,g)$.

Put
$$U_0(27)/U=\langle\epsilon\rangle,\quad \epsilon=\left (\begin{array}{cc}1&0\\0&-1\end{array}\right ).$$
The non-trivial Galois action of $\Gal(K/\BQ)$ on $X_U$ is given by
the right translation of $\epsilon$ on $X_U$. We have
\[\Aut_\BQ(X_U)=\Aut_K(X_U)\rtimes\Gal(K/\BQ)\simeq (X_U(K)\rtimes \CO_K^\times)\rtimes \Gal(K/\BQ).\]
Let $N_{\GL_2(\BA_f)}(U)$ be the normalizer of $U$ in
$\GL_2(\BA_f)$. Then there is a natural homomorphism
$$ N_{\GL_2(\BA_f)}(U)/U\longrightarrow \Aut_\BQ(X_U)$$ induced by the right
translation on $X_U$. The curve $X_U$ is not geometrically connected
and has two connected components over $\BC$. An element $g\in
N_{\GL_2(\BA_f)}(U)$ maps one component of $X_U$ onto the other if and
only if it has image $-1$ under the composition of the following
morphisms:
\[\xymatrix{\GL_2(\BA_f)=\GL_2(\BQ)^+\GL_2(\wh{\BZ})\ar[r]^{\quad\quad \quad\quad \quad  \det}&\BQ^{\times}_+\widehat{\BZ}^\times\ar[r]&\BZ_3^\times/(1+3\BZ_3),}\]
where the last morphism is the projection from
$\widehat{\BZ}^\times$ to its $3$-adic factor.

Let $p\equiv 2,5\mod 9$ be a rational odd prime number. Put $$\tau=M\omega=\omega p/9,\quad M=\matrixx{p}{0}{0}{9}.$$
Let $\rho:K\rightarrow \M_2(\BQ)$ be the normalized embedding with a fixed point $\tau\in \CH$, i.e. we have
\[\rho(t)\begin{pmatrix}\tau\\1\end{pmatrix}=t\begin{pmatrix}\tau\\1\end{pmatrix},\quad \textrm{for any $t\in K$}.\]
Then the embedding $\rho:K\rightarrow\M_2(\BQ)$ is explicitly given by
\[\rho(\omega)=M\matrixx{-1}{-1}{1}{0}M^{-1}=
\begin{pmatrix} -1 & -p/9\\
9/p & 0\end{pmatrix}.\]
Let $R_0(27)$ be the standard Eichler order of discriminant $27$ in $\M_2(\BQ)$, i.e.
\[R_0(27)=\matrixx{\BZ}{\BZ}{27\BZ}{\BZ}.\]  Then $K\cap R_0(27)=\CO_{9p}$. Let $\CO_{K,3}$ be the completion of $\CO_K$ at the unique place above $3$. We have
\[ \CO_{K,3}^\times/\BZ_3^\times(1+9\CO_{K,3})=\langle \omega_3\rangle^{\BZ/3\BZ}\times\langle1+3\omega_3\rangle^{\BZ/3\BZ},\]
where $\omega_3$ is the $3$-local component of $\omega$.
It is straightforward to verify that $1+3\omega_3$ normalizes $U$, and hence it induces an automorphism of $X_U$.

\begin{thm}
For any point $P\in X_U$, we have
\[P^{1+3\omega_3}=[\omega^2]P.\]
\end{thm}
\begin{proof}
Since $1+3\omega_3$ has determinant $\equiv 1\mod 3$, as an element in $\Aut_\BQ(X_U)$, they lie in the subgroup $\Aut_K(X_0(27))$.
Suppose $P=[z,1]$, $z\in \CH$, be an arbitrary point on $X_0(27)$. We have
\[BC^2(1+3\omega_3)=\left(\begin{pmatrix} 45/p - 38& -19p/3 + 5/3\\
513/p - 432 & -72p + 19\end{pmatrix}_3, BC^2\right)\in U,\]
where the subscript $3$ denotes the $3$-adic component of the adelic matrices.
Then
\[P^{1+3\omega_3}=\Phi(BC^2)(P)=[\omega^2]P.\]
\end{proof}

Let $\sigma : \wh{K}^\times \ra \Gal(K^\ab/K)$ be the Artin reciprocity map and we denote by $\sigma_t$ the image of $t\in \wh{K}^\times $. Let $P_\tau=[\tau,1]_{U_0(27)}$ be the CM point on $X_0(27)$ arising from $\tau$.
\begin{coro}\label{galois}
The point $P_\tau\in X_0(27)$ is defined over $H_{9p}$ and 
\[P_\tau^{\sigma_{1+3\omega_3}}=[\omega^2]P_\tau.\]
\end{coro}
\begin{proof}
By Shimura's reciprocity law,
\[P_{\tau,U}^{\sigma_t}=P_{\tau,U}^t=[\tau,t]_U,\quad t\in \wh{K}^\times.\]
Since $\wh{K}^\times\cap U=\wh{\CO_{9p}}^\times$, we see that $P_{\tau,U}$ is defined over the ring class field $H_{9p}$. Also by Shimura's reciprocity law, 
\[P_{\tau,U}^{\sigma_{1+3\omega_3}}=P_{\tau,U}^{1+3\omega_3}=[\omega^2]P_{\tau,U},\]
and the statements for $P_\tau\in X_0(27)$ follow from the natural projection $X_U\ra X_0(27)$.

\end{proof}

\begin{prop}\label{LCF}
Let $p\equiv 2,5\mod 9$ be odd primes.
\begin{itemize}
\item[1.] The field $H_{9p}=H_{3p}(\sqrt[3]{3})$ with Galois group $\Gal(H_{9p}/H_{3p})\simeq \langle1+3\omega_3\rangle^{\BZ/3\BZ}$, and
\[\left(\sqrt[3]{3}\right)^{\sigma_{1+3\omega_3}-1}=\omega^2.\]
\item[2.] The field $K(\sqrt[3]{p})$ is contained in $H_{3p}$. 
\item[3.] The field extension $H_{3p}/K$ is totally ramified at $p$ and $H_{9p}/H_{3p}$ is split at the places above $p$.
\item[4.] We have $\left(\sqrt[3]{3}\right)^{\sigma_{\omega_3}-1}=1$ and
\[\left(\sqrt[3]{p}\right)^{\sigma_{\omega_3}-1}=\left\{\begin{aligned} {\omega},&\quad p\equiv 2\mod 9 \\{\omega^2},&\quad p\equiv 5\mod 9  \end{aligned}\right.\]
\end{itemize}
\end{prop}
\begin{proof}
For the first assertion, the Galois group
\[\Gal(H_{9p}/H_{3p})\simeq K^\times \wh{\CO_{3p}}^\times/K^\times\wh{\CO_{9p}}^\times\]
is cyclic of order $3$ and generated by $1+3\omega_3$. The ideal $7\CO_K=(1+3\omega)(1+3\omega^2)$ and let $v$ be the place corresponding to the prime ideal $(1+3\omega)$. Then by the local-global principle, we have
\[\left(\sqrt[3]{3}\right)^{\sigma_{1+3\omega_3}-1}=\hilbert{1+3\omega_3}{ 3}{K_3}{3}=\hilbert{1+3\omega_v}{ 3}{K_v}{3}^{-1}=3^{-2}\mod (1+3\omega)=\omega^2,\]
where $\hilbert{\cdot}{\cdot}{K_w}{3}$ denotes the $3$-rd Hilbert symbol over $K_w$.

For the second assertion, by the class field theory, it suffices to prove that, under the Artin reciprocity map, 
$$U_{3p}=\wh{\BZ}^\times(1+3\CO_{K,3})(1+p\CO_{K,p})\prod_{v\nmid 3p}\CO_{K,v}^\times$$ 
fixes $\sqrt[3]{p}$. Since $K(\sqrt[3]{p})$ is anticyclotomic over $K$, $\wh{\BZ}^\times$ fixes $\sqrt[3]{p}$.  Since $K(\sqrt[3]{p})/K$ is unramified outside $3p$, $\prod_{v\nmid 3p}\CO_{K,v}^\times$  fixes $\sqrt[3]{p}$. Using the Hilbert symbol, it is clear that $(1+p\CO_{K,p})$ fixes $\sqrt[3]{p}$. Finally, we look at the $3$-adic place. Since $1+9\CO_{K,3}\subset (K^\times_3)^3$, it suffices to prove $1+3\omega_3$ fixes $\sqrt[3]{p}$. We have
\[(\sqrt[3]{p})^{\sigma_{1+3\omega_3}-1}=\hilbert{1+3\omega_3}{p}{K_3}{3}.\]
If $p\equiv 2\mod 9$, then
\[(\sqrt[3]{p})^{\sigma_{1+3\omega_3}-1}=\hilbert{1+3\omega_3}{2}{K_3}{3}=\hilbert{1+3\omega_2}{2}{K_2}{3}^{-1}\hilbert{1+3\omega_v}{2}{K_v}{3}^{-1}=1.\]
Similarly, if $p\equiv 5\mod 9$, we also have $(\sqrt[3]{p})^{\sigma_{1+3\omega_3}-1}=1$.

For the third assertion, let $v$ be a place of $H_{3p}$ above $p$ and consider the following exact sequence
\[0\ra I_v\ra \Gal(H_{3p,v}/K_p)\ra \Gal(k_v/\BF_{p^2})\ra 0\]
where $k_v$ is the residue field of $H_{3p,v}$ and $I_v$ is the inertia group. Under the  class field theory, we have the following commutative diagram
\[\xymatrix{\Gal(H_{3p,v}/K_p)\ar[d]\ar[r]^{\simeq}& K_p^\times / \N (H_{3p,v}^\times)\ar[d]\\\Gal(H_{3p}/K)\ar[r]^{\simeq\quad }& \wh{K}^\times / K^\times \wh{\BZ}^\times(1+3p\wh{\CO_{K}})}\]
with vertical inclusions. Under the upper  isomorphism, we have
\[I_v\simeq \CO_{K,p}^\times / (\CO_{K,p}^\times \cap K^\times \wh{\BZ}^\times(1+3p\wh{\CO_{K}})).\]
Note
\[\CO_{K,p}^\times \cap K^\times \wh{\BZ}^\times(1+3p\wh{\CO_{K}})=\CO_K^\times \BZ_p^\times(1+p\CO_{K,p}).\]
On the other hand, 
 $$\Gal(H_{3p}/K)\cong \wh{K}^\times / K^\times \wh{\BZ}^\times(1+3p\wh{\CO_{K}})\simeq  \CO_{K,p}^\times/\CO_{K}^\times\BZ_p^\times(1+p\CO_{K,p}).$$
Hence 
\[\Gal(H_{3p}/K)=\Gal(H_{3p,v}/K_p)=I_v,\]
and thus $H_{3p}/K$ is totally ramified at $p$. Since $p\equiv 2\mod 3$, we see that $x^3-3\mod p$ has a root in $\BF_p^\times$, and hence by Hensel's lemma $x^3-3$ has a root in $\BQ_p$. Then we see that $\sqrt[3]{3}\in K_{p}$ and $p$ is split in
$K(\sqrt[3]{3}/K)$. Then it follows that $H_{9p}/H_{3p}$ is split at the places above $p$.

For the last assertion. Since $p\equiv 2\mod 3$, the prime $p$ is inert in $K$. Then 
\[\left(\sqrt[3]{p}\right)^{\sigma_{\omega_3}-1}=\hilbert{\omega_p}{p}{K_p}{3}^{-1}=\omega^{\frac{p^2-1}{3}}.\]
\end{proof}

\section{Heegner Points}
\subsection{A Kronecker congruence for supersingular points}
Let $j(z)$ be the $j$-function on the Poinc\'are upper half plane $\CH$.  For any integer $N\geq 1$, denote $j_N(z)=j(Nz)$. Then there is a polynomial $\Phi_N(X,Y)\in \BZ[X,Y]$, called the modular equation, such that
\[\Phi_N(j,j_N)=0.\]
The classical Kronecker congruence \cite[\S 11-C]{Cox89} states that for a prime $p$
\[\Phi_p(X,Y)\equiv (X-Y^p)(X^p-Y)\mod p\BZ[X,Y].\]
Let $K\subset \BC$ be an imaginary quadratic field, and let $\tau\in \CH\cap K$.
Then by CM theory, both $j(\tau)$ and $j(p\tau)$ are algebraic integers in some abelian extension of $K$. Then the Kronecker congruence states either
\[j(\tau)\equiv j(p\tau)^p\mod p\textrm{ or } j(p\tau)\equiv j(\tau)^p\mod p.\]

Dasgupta and Voight \cite[Proposition 5.2.1]{DV17} give a refined generalization of Kronecker congruence to modular forms when $p$ is split in the imaginary quadratic field $K$. In this subsection we give a refined generalization of Kronecker congruence to modular functions when $p$ is inert in $K$.

For $N\geq 4$, the modular curve $Y_0(N)$ has a smooth affine model over $\BZ$ which represents the functor "isomorphism classes of elliptic curves with $\Gamma_0(N)$-structures" \cite[\S 2.5]{Katz76}. Suppose $p\nmid N$. We have two degeneracy maps $\alpha:Y_0(Np)\ra Y_0(N)$ and $\beta:Y_0(Np)\ra Y_0(N)$ given as
\[\alpha(E,H,C)=(E,H),\quad \beta(E,H,C)=(E/C,(H\oplus C)/C),\]
where $H$ and $C$ are cyclic subgroups of the elliptic curve $E$ of order $N$ and $p$ respectively. Recall that we have the Atkin-Lehner operator $W_p$ at $p$:
\[W_p(E,H,C)=(E/C,(H\oplus C)/C,E[p]/C).\]
It follows that
\[\beta=\alpha \circ W_p.\]

The modular curve $Y_0(N)_{/\BZ[1/N]}$ is finite flat over the $j$-line $\BA^1_{/\BZ[1/N]}$, and the modular curve $Y_0(N)_{/\BQ}$ has function field $\BQ(j,j_N)$. Let $\CO(N)$ be the integral closure of $\BZ[1/N,j]$ in $\BQ(j,j_N)$. Then $\CO(N)$ is the affine coordinate ring of $Y_0(N)_{/\BZ[1/N]}$. 

\begin{prop}\label{prop: Kronecker}
Let $\tau\in\CH$ be imaginary quadratic and let $p\nmid N$ be a prime inert in the imaginary quadratic field $\BQ(\tau)$. Suppose $f\in \CO(N)$ is a modular function integral over $\BZ[1/N,j]$. Then both $f(p\tau)$ and $f(\tau)$ are $p$-adic integers and 
\[f(p\tau)\equiv f(\tau)^p\mod p.\]
\end{prop}
\begin{proof}
Since $f$ is integral over $\BZ[1/N,j]$ and $j(\tau), j(p\tau)$ are algebraic integers, we see $f(\tau),f(p\tau)$ are $p$-adic integers. Since $f$ is a regular function on $Y_0(N)_{/\BZ[1/N]}$, $f$ reduces to a regular function $\wt{f}$ on $Y_0(N)_{/\BF_p}$. 

It follows from \cite{DR73} (see also \cite[Proposition 3.8 (i)]{Ribet90}) that the Atkin-Lehner operator $W_p$ acts on the set of supersingular points on $X_0(Np)_{/\BF_p}$ as the Frobenius automorphism $\Frob_p$. For any supersingular point $x$ on $X_0(Np)_{/\BF_p}$, we have
\[\wt{f}(px)=\beta^*(\wt{f})(x)=W_p^*\circ \alpha^*(\wt{f})(x)=\alpha^*(\wt{f})(W_p(x))=\alpha^*(\wt{f})(x^{\Frob_p})=\left(\wt{f}(x)\right)^p.\]

Since $p$ is inert in $\BQ(\tau)$, the CM point arising from $\tau$ lies in the supersingular locus on $X_0(Np)_{/\BF_p}$ and the proposition follows. 

\end{proof}

\subsection{Nontriviality of Heegner points}
We apply the generalized Kronecker congruence Proposition \ref{prop: Kronecker} to some explicit modular functions and these congruences are crucial to investigate the  Heegner points under reduction modulo $p$.
\begin{prop}\label{cong}
\begin{itemize}
\item[1.] The function $f(z)=\eta(27z)/\eta(3z)$ is a modular function on $\Gamma_0(81)$ and $g(z)=(\eta(9z)/\eta(3z))^{4}$  is a modular function on $\Gamma_0(27)$.
\item[2.] Let $p$ be a prime number inert in $K=\BQ(\omega)$. Then the values $f(\omega/9)$, $g(\omega/9)$, $f(\omega p/9)$ and $g(\omega p/9)$ are all $p$-adic units and
\[f(\omega p/9)\equiv f(\omega/9)^p\mod p \textrm{ and }g(\omega p/9)\equiv g(\omega/9)^p\mod p.\]
\end{itemize}
\end{prop}
\begin{proof}
By Ligozat's criterion \cite[Proposition 3.2.1]{Ligozat1970}, $f$ is a modular function on $\Gamma_0(81)$ and $g$ is a modular function on $\Gamma_0(27)$. 

By \cite[Chapter 12, Theorem 2]{Lang-ef}, 
\[(\det \alpha )^{12} f^{24}=(\det \alpha )^{12} \frac{\Delta(\alpha(3z))}{\Delta(3z)},\quad \alpha=\matrixx{9}{0}{0}{1},\]
is integral over $\BZ[j_3]$. Hence $3f$ is integral over $\BZ[j_3]$ and hence integral over $\BZ[j]$. Taking $z_0=\omega/9$ and $\omega p/9$ respectively, by \cite[Chapter 12, Theorem 4]{Lang-ef}, we see that $3 f(z_0)$ is an algebraic integer dividing $3$. In particular, both $f(\omega/9)$ and $f(\omega p/9)$ are $p$-adic units. Then apply Proposition \ref{prop: Kronecker} to $3f$ we get
 \[f(\omega p/9)\equiv f(\omega/9)^p\mod p.\]
 Viewing $g$ as a modular function on $\Gamma_0(81)$, the congruence relation for $g$ follows similarly.
\end{proof}

Recall from Proposition \ref{mc} that the elliptic curve $E_1:y^2=x^3-432$ has modular parametrization
by the formulae
\[x(z)=4\frac{\eta(9z)^4}{\eta(3z)\eta(27z)^3}=4gf^{-3}\]
and
\[y(z)=8\frac{\eta(3z)}{\eta(27z)}+36=8f^{-3}+36.\]

\begin{coro}
The coordinates $x(\omega p/9)$ and $y(\omega p/9)$ are $p$-adic integers and satisfy the congruence relations
\[x(\omega p/9)\equiv -4\cdot 3^{7/6}e^{5\pi i/6}\mod p\]
and
\[y (\omega p/9)\equiv -8\cdot 3\sqrt{-3}+36\mod p.\]
\end{coro}
\begin{proof}
By Proposition \ref{cong}, we know that all the coordinates $f(z_0)$ and $g(z_0)$, with $z_0 =\omega/9$ or $\omega p/9$, are $p$-adic units, and the congruence relations
$$gf^{-3}(\omega p/9)\equiv \left(gf^{-3}(\omega/9)\right)^p\mod p,\quad f^{-3}(\omega p/9)\equiv \left(f^{-3}(\omega/9)\right)^p\mod p.$$

Put $X=gf^{-3}$ and $Y=f^{-3}$. In terms of coordinates in $X$ and $Y$,  the equation of the elliptic curve $E_1$  is given by $Y^2+9Y=X^3-27$. It remains to compute the values $ X(\omega/9)=gf^{-3}(\omega/9)$ and $Y(\omega/9)=f^{-3}(\omega/9)$ explicitly. By the proof of \cite[Lemma 5.2.4]{DV17}, we see that
\[f(\omega/9)=e^{-\pi i/6}\frac{1}{\sqrt{3}},\]
and hence $Y(\omega /9)=3\sqrt{-3}$, and then we find $X(\omega/9)=3^{7/6}e^{\pi i/6}$.

Since $(-3)^{\frac{p-1}{2}}=\left(\frac{-3}{p}\right)=-1\mod p$,
\[Y(\omega p/9)\equiv (3\sqrt{-3})^p\equiv 3(-3)^{\frac{p-1}{2}}\sqrt{-3}\equiv -3\sqrt{-3}\mod p.\]
We concludes from the Weierstrass euqation $E_1: Y^2+9Y=X^3-27$  and
\[X(\omega p/9)^2\equiv 3^{7p/3}e^{p\pi i/3}\equiv -3^{7/3}e^{2\pi i/3}\mod p\]
that
\[X(\omega p/9)\equiv -3^{7/6}e^{5\pi i/6}\mod p.\]
Then the congruence relations for $x(\omega p/9)$ and $y(\omega p/9)$ follow.
\end{proof}

Before we prove the nontriviality of Heegner points, we need to analyze the behavior of points over finite fields. Consider the elliptic curve $E_9: y^2=x^3-2^4\cdot 3^7$ and let $\varphi:E_1\ra E_9$ be the isomorphism given by $$(x,y)\mapsto(3^{4/3}x,9y).$$
Note that $P_0=(36,108)$ is an infinite-order point in $E_9(\BQ)$ and $E_9(\BQ)=\BZ \cdot P_0$. Since $p \equiv 2,5\mod 9$ is inert in $K$, $E_9$ has good supersingular reduction modulo $p$. Then $E_9(\BF_p)$ has cardinality $p+1$, and its 3-primary component is isomorphic to $\BZ/3\BZ$.

Fix an embedding of $\ov{\BQ}$ the field of algebraic numbers into the complex number field $\BC$.
\begin{prop}\label{finite}
\begin{itemize}
\item[1.] Let $\zeta_9=e^{2\pi i/9}$ be a primitive $9$-th root of unity and let $A\in E_9(\ov{\BQ})$ such that $[3]A=P_0$. Then $\BQ(\zeta_9)\subset \BQ(A,E_9[3])$.
\item[2.] The point $P_0$ is not divisible by $3$ in $E_9(K_p)$. 
\item[3.] The reduction of $P_0$ in $E_9(\BF_p)$ has non-trivial 3-component, i.e. the reduction of $[(p+1)/3] P_0$ is nonzero.
\end{itemize}
\end{prop}
\begin{proof}
Let $A=(x_0,y_0)\in E_9(\ov{\BQ})$ and $[3]A=P_0$. By the triplication formula \cite[page 105]{Silvermanbook1}, we have
\[x(P_0)=x([3]A)=\frac{x_0^9+2^93^8x_0^6+2^{12}3^{15}x_0^3-2^{18}3^{21}}{(3x_0(x_0^3-2^63^7))^2}=36.\]
The $x_0$ is a root of a polynomial $f(x)$ of degree $9$, where explicitly
\[f(x)=x^9 - 2^2 3^4x^8 + 2^9  3^8x^6 + 2^9  3^{11}x^5 + 2^{12} 3^{15}x^3 - 2^{14} 3^{18}x^2 - 2^{18} 
3^{21}.\]

Suppose
\[A+(0,2^23^3\sqrt{-3})=(x_1,y_1)\textrm{ and }A+(0,-2^23^3\sqrt{-3})=(x_2,y_2).\]
Then
\[x_0+x_1=\left(\frac{y_0-2^23^3\sqrt{-3}}{x_0}\right)^2\]
and 
\[x_0+x_2=\left(\frac{y_0+2^23^3\sqrt{-3}}{x_0}\right)^2.\]
It follows that
\[t=x_0+x_1+x_2=\frac{x_0^3-2^63^7}{x_0^2}.\]
It is straightforward to verify that $t$ satisfies a polynomial of degree $3$
\[g(t)=t^3-2^23^4t^2+2^63^{10}=f(x_0)/x_0^6=0.\]
By the cubic root formula, we see that $g(t)$ has a root
\[t=108(\zeta_9+\zeta_9^{-1}+1).\]
It is easy to see that $K=\BQ(\sqrt{-3})\subset \BQ(A,E_9[3])$. Since $\BQ(A,E_9[3])$ is Galois over $\BQ$, we see
\[\BQ(\zeta_9)=\BQ(\zeta_9+\zeta_9^{-1},\sqrt{-3})\subset \BQ(A,E_9[3]).\]
The first assertion follows.

For the second assertion, we assume the contrary, i.e. there exists some $A\in E_9(K_p)$ with $[3]A=P_0$. Note $E_9[3]\subset E(K_p)$. Indeed, we have
\[E_9[3]=\{O,(0,\pm 2^23^3\sqrt{-3}),(2^23^2\sqrt[3]{3}\omega^i,\pm 2^23^4)\}_{i=0,1,2},\]
and $\sqrt[3]{3}\in K_p$ because $p\equiv 2\mod 3$. Then
\[\BQ(A,E_9[3])\subset K_p,\]
and it follows that $p$ is split in the field extension $\BQ(A,E_9[3])/K$.

On the other hand, since $p\equiv 2,5\mod 9$, $p$ is a primitive root modulo $9$. Then the Frobenius automorphism $\Frob_p$ is the generator of the Galois group 
\[\Gal(\BQ(\zeta_9)/\BQ)\simeq (\BZ/9\BZ)^\times.\]
Then $p$ is inert in the field extension $\BQ(\zeta_9)/\BQ$, and in particular, $p$ is inert in the field extension $\BQ(\zeta_9)/K$. Then a contradiction follows by the first statement that $\BQ(\zeta_9)\subset \BQ(A,E_9[3])$, and hence $P_0$ is not divisible by $3$ in $E_9(K_p)$.

For the third assertion, consider the following commutative diagram
\[\xymatrix{0\ar[r]&\wh{E_9}(p\BZ_p)\ar[d]^{[3]}\ar[r]&E_9(\BQ_p)\ar[r]\ar[d]^{[3]}&E_9(\BF_p)\ar[d]^{[3]}\ar[r]&0\\
0\ar[r]&\wh{E_9}(p\BZ_p)\ar[r]&E_9(\BQ_p)\ar[r]&E_9(\BF_p)\ar[r]&0,}\]
where $\wh{E_9}$ denotes the formal group associated to $E_9$.
Since $[3]$ is an isomorphism on $\wh{E_9(p\BZ_p)}$, we conclude that the reduction modulo $p$ induces an isomorphism
\[E_9(\BQ_p)/3E_9(\BQ_p)\simeq E_9(\BF_p)/3E_9(\BF_p)\simeq \BZ/3\BZ.\]
By the second assertion, we see that $P_0$ is not divisible by $3$ in $E_9(\BQ_p)$, and hence it gives rise to a nonzero element in $E_9(\BF_p)/3E_9(\BF_p)$ under reduction modulo $p$. The third assertion follows.
\end{proof}
Recall that $\tau=\omega p/9$ and $P_\tau=[\tau,1]\in X_0(27)(H_{9p})$ and put
\[z=\Tr_{H_{9p}/L_{(3,p)}}P_\tau.\]
Note that $H_{3p}/K$ is totally ramified at $p$ and $H_{9p}/H_{3p}$ is split at the places above $p$. Let $\fP$ be a place of $L_{(3,p)}$ above $p$.
\begin{thm}\label{thm:1}
\begin{itemize}
\item[1.] The point $z\mod \fP$ is primitive of order $3$ in $E_1(\BF_{p^2})$, i.e. not killed by $[\sqrt{-3}]$.
\item[2.] The point $z\in E_1(L_{(3,p)})$ is of infinite order.
\end{itemize}
\end{thm}

\begin{proof}
Note that $$P=(-4\cdot 3^{7/6}e^{5\pi i/6},-8\cdot 3\sqrt{-3}+36)\in E_1(\ov{\BQ}),$$
and then
\[P\equiv  P_\tau\mod p.\]
Put
\[P'=[\omega^2]P+(0,-12\sqrt{-3})=(4\cdot 3^{2/3},12),\]
and let $\varphi:E_1\ra E_9$ be the isomorphism given by $$(x,y)\mapsto(3^{4/3}x,9y).$$
Then  $\varphi(P')=(36,108)$ is the generator of $E_9(\BQ)$.   Note
\[E_9[\sqrt{-3}]=\{O, (0,\pm 12\sqrt{-3})\}\subset E_9(\BF_{p^2}).\] 
By Proposition \ref{finite}, we see that $[(p+1)/3]\varphi(P')$ is a nonzero point of order $3$ in $E_9(\BF_p)$ which is not killed by $[\sqrt{-3}]$.

Then we know that $[(p+1)/3]P'\mod \fP$ is primitive of order $3$, and hence
\[z\equiv  \left[\frac{p+1}{3}\right] P_\tau \mod \fP\in E_1(\BF_{p^2})[3]\]
is primitive of order $3$.

For the second assertion. If $z$ is torsion, then by Proposition \ref{torsion},   $z\in E_1(K)$ and  hence
\[z^{\sigma_{1+\omega_3}}=z=[\omega^2]z.\]
Then $z$ is killed by $[\sqrt{-3}]$ and so is the reduction of $z\mod \fP$, which is a contradiction.
\end{proof}

Let $E$ be an elliptic curve over $\BQ$ which has CM by $\CO_K$ over $\ov{\BQ}$. Let $L/K$ be a cubic field extension and $\chi:\Gal(L/K)\ra \CO_K^\times$ a nontrivial cubic character.    For any abelian group $M$, denote $M_\BQ=M\otimes_\BZ\BQ$. Put
\[E(L)^\chi=\{P\in E(L):\,P^\sigma =[\chi(\sigma)]P,\forall \sigma\in \Gal(L/K).\}\]

\begin{lem}\label{lem:2}
 Let the notations be as above. Then we have
$$E(L)_\BQ=E(K)_\BQ\oplus E(L)_\BQ^\chi\oplus E(L)_\BQ^{\chi^{-1}}.$$
\end{lem}

\begin{proof}
For any $\alpha=P\otimes 1\in E(L)_\BQ$, define for $k=0,1,2$,
\begin{eqnarray*}
\alpha_k=\frac{1}{3}\left (\sum_{\sigma\in G}[\chi^{-k}(\sigma)]P^\sigma\otimes 1\right )\in E(L)_\BQ^{\chi^{k}}.
\end{eqnarray*}
Then
$P\otimes 1=\sum_{k} \alpha_k$ and hence $E(L)_\BQ=E(K)_\BQ+E(L)_\BQ^{\chi}+E(L)_\BQ^{\chi^{2}}$.
\par Let $\alpha_i\in E(L)^{\chi^i}_\BQ$, $i=0,1,2$, satisfy $\alpha_0+\alpha_1+\alpha_2=0$. We will prove that $\alpha_i=0$, for all $i$. By multiplying a sufficiently large integer, we may assume that $\alpha_i=P_i\otimes 1$, $P_i\in E(L)$. Let $\tau$ be a generator of $G$. Then we have for $k=0,1,2$,
\[A_k=\left (\sum_{i=0}^2P_i\otimes 1\right )^{\tau^k}=\sum_{i=0}^2 [\chi^i(\tau^k)]P_i\otimes 1=0.\]
Then
\[3 P_0\otimes 1=\sum_{k=0}^2 A_k=0 .\]
Then $P_1\otimes 1+P_2\otimes 1=0$. Hence, $P_1\otimes 1\in E(L)_\BQ^\chi\cap E(L)_\BQ^{\chi^2}$ and then
\[[\chi(\sigma)]P_1\otimes 1=[\chi^2(\sigma)]P_1\otimes 1, \forall \sigma\in G.\]
In particular, $[\omega]P_1\otimes 1=[\omega^2]P_1\otimes 1$. Thus $3P_1\otimes 1=0$ and then for all $i$, $\alpha_i=0$.

\end{proof}

Now we are at the stage to prove the main theorem.
\begin{proof}[Proof of the Main Theorem \ref{main}]
The elliptic curve $E_3$ has the Weierstrass equation $y^2=x^3-2^4\cdot 3^5$ and let $\phi: E_1\ra E_3$ be the isomorphism given by
\[\phi(x,y)=(3^{2/3}x,3y).\]
Then $\phi:E_1(L_{(3,p)})^{\sigma_{1+3\omega_3}=\omega^2}\simeq E_3(L_{(p)})$ is an isomorphism. Note $E_3(K)_\BQ=0$. By Lemma \ref{lem:2}, we have a decomposition in $E_1(L_{(3,p)})_{\BQ}$
\begin{equation}\label{decom1}
z=z_1+z_2, \quad z_1,z_2, \in E_1(L_{(3,p)})^{\sigma_{1+3\omega_3}=\omega^2}_\BQ
\end{equation}
such that  $\phi(z_1)\in E_3(L_{(p)})^{\chi_p}_\BQ$, $\phi(z_2)\in E_3(L_{(p)})^{\chi_p^{-1}}_\BQ$. 

By Theorem \ref{thm:1}, the point $z$ is of infinite order. Then either $z_1$ or $z_2$ is nonzero.  If $z_1$ (resp. $z_2$) is nonzero, then $\phi(z_1)$  (resp. $\phi(z_2)$) gives rise to a non-torsion point in $E_3(L(p))^{\chi_p}\simeq E_{3p}(K)$ (resp. $E_3(L(p))^{\chi_p^{-1}}\simeq E_{3p^2}(K)$). Since both $E_{3p}$ and $E_{3p^2}$ have CM by $K$, we conclude that
\[\rk_{\BZ} E_{3p}(\BQ)=\rk_{\CO_K} E_{3p}(K)>0\ \  (\textrm{resp.  }\rk_{\BZ} E_{3p^2}(\BQ)=\rk_{\CO_K} E_{3p^2}(K)>0).\]
Hence, either $3p$ or $3p^2$ is a cube sum.
\end{proof}

\section{The Explicit Gross-Zagier Formulae}
\subsection{Test vectors and the explicit Gross-Zagier formulae}
Let $\pi$ be the automorphic representation of $\GL_2(\BA)$ corresponding to ${E_1}_{/\BQ}$. Then $\pi$ is only ramified at $3$ with conductor $3^3$. For $n\in \BQ^\times$, let $\chi_n: \Gal(K^{\ab}/K)\rightarrow\BC^\times$
be the cubic character given by $\chi_n(\sigma)=(\sqrt[3]{n})^{\sigma-1}$. Define
\[L(s,E_1,\chi_n)=L(s-1/2,\pi_K\otimes \chi_n),\quad \epsilon(E_1,\chi_n)=\epsilon(1/2,\pi_K\otimes \chi_n),\]
where $\pi_K$ is the base change of $\pi$ to $\GL_2(\BA_K)$.

Let $p\equiv 2,5\mod 9$ be an odd prime number, and put $\chi=\chi_{3p}$ or $\chi_{3p^2}$. By \cite[Proposition 4.1]{SY17}, we have
\[L(s,E_1,\chi)=\left\{\begin{aligned}
L(s,E_{3p})L(s,E_{9p^2})&\quad \textrm{ if $\chi=\chi_{3p}$,}\\L(s,E_{3p^2})L(s,E_{9p})&\quad \textrm{ if $\chi=\chi_{3p^2}$.}
\end{aligned}\right.\]

By \cite{Liverance}, we have  the epsilon factors $\epsilon(E_{3p})=\epsilon(E_{3p^2})=-1$ and $\epsilon(E_{9p})=\epsilon(E_{9p^2})=+1$, and hence the epsilon factor $\epsilon(E_1,\chi)=-1$. For a quaternion algebra $\BB_{/\BA}$, we define its ramification index $\epsilon(\BB_v)=+1$ for any place $v$ of $\BQ$ if the local component $\BB_v$ is split and $\epsilon(\BB_v)=-1$ otherwise.
\begin{prop}\label{Tun-Saito}
The incoherent quaternion algebra $\BB$ over $\BA$, which satisfies
$$\epsilon(1/2,\pi_v,\chi_v)=\chi_{v}(-1)\epsilon_v(\BB)$$
for all places $p$ of $\BQ$, is only ramified at the infinity place.
\end{prop}
\begin{proof} Since $\pi$ is unramified at finite places $v\nmid 3$ and $\chi$ is unramified at finite places $v\neq 3p$ and $p$ is inert in $K$, by \cite[Proposition 6.3]{Gross88}, we get $\epsilon(1/2,\pi_v,\chi_v)=+1$ for all  finite $v\neq 3$. Again by \cite[Proposition 6.5]{Gross88}, we also know that $\epsilon(1/2,\pi_\infty,\chi_\infty)=-1$. Since $\epsilon(1/2,\pi,\chi)=-1$, we see that $\epsilon(1/2,\pi_3,\chi_3)=+1$. Since $\chi$ is a cubic character, $\chi_v(-1)=1$ for any $v$. Hence $\BB$ is only ramified at the infinity place.
\end{proof}
Let $\BB_f^\times=\GL_2(\BA_f)$ be the finite part of $\BB^\times $. For any open compact subgroup $U\subset \BB_f^\times$,  the Shimura curve $X_U$  associated to $\BB$ of level $U$ is the usual modular curve with complex uniformazition
$$X_U(\BC)=\GL_2(\BQ)^+\backslash\left(\CH \bigsqcup \BP^1(\BQ)\right)\times \GL_2(\BA_f)/U.$$
Let $$\pi_{E_1}=\varinjlim_{U}\Hom^0_{\xi_U}(X_U,E_1),$$
where $\Hom^0_{\xi_{U}}(X_U,E_1)$ denotes the morphisms in $\Hom_\BQ(X_U,E_1)\otimes_\BZ\BQ$ using the Hodge class $\xi_U$ as a base point. Then $\pi_{E_1}$ is an automorphic representation of $\BB^\times$ over $\BQ$ and $\pi$ is the Jacquet-Langlands correspondence of $\pi_{E_1}\otimes_\BQ\BC$ on $\GL_2(\BA)$. By Proposition \ref{Tun-Saito} and a theorem of Tunell-Saito \cite[Theorem 1.4.1]{YZZ}, the space
\[\Hom_{\BA_K^\times}(\pi_{E_1}\otimes \chi,\BC)\otimes \Hom_{\BA_K^\times}(\pi_{E_1}\otimes \chi^{-1},\BC)\]
is one-dimensional with a canonical generator $\beta^0=\otimes \beta_v^0$ where, for each place $v$ of $\BQ$, the bilinear form
\[\beta_v^0:\pi_{E_1,v}\otimes \pi_{E_{1,v}}\lra \BC\]
is given by
\[\beta_v^0(f_1\otimes f_2)=\int_{K_v^\times/\BQ_v^\times}\frac{(\pi_{E_1,v}(t)f_1,f_2)\chi_v(t)}{(f_1,f_2)_v}dt,\quad f_1,f_2\in \pi_{E_1,v}.\]
Here $(\cdot,\cdot)_v$ is a $\BB_v^\times$-invariant pairing on $\pi_{E_1,v}\times \pi_{E_1,v}$ and $dt$ is a Haar measure on $K_v^\times/\BQ_v^\times$. For more details we refer to \cite[Section 1.4]{YZZ} and \cite[Section 3]{CST14}.

The elliptic curve $E_{1}$ has conductor $3^3$ and let $f:X_0(3^3)\ra E_1$ be the modular  isomorphism given by Proposition \ref{mc}. Let $$\CR=\matrixx{\widehat{\BZ}}{\widehat{\BZ}}{3^3\cdot\widehat{\BZ}}{\widehat{\BZ}} \subset \BB_f(\wh{\BZ})=\M_2(\wh{\BZ})$$
be the Eichler order of discriminant $3^3$.
Then $U_0(3^3)=\CR^\times$, and by the newform theory \cite{Casselman1973}, the invariant subspace $\pi_{E_1}^{\CR^\times}$ has dimension $1$ and is generated by $f$.

\begin{prop}\label{test-vector}
The modular isomorphism $f:X_0(3^3)\ra E_1$ is a test vector for the pair $(\pi_{E_1},\chi)$, i.e. $\beta^0(f,f)\neq 0$.
\end{prop}
\begin{proof}
Let $\CR'$ be the admissible order for the pair $(\pi_{E_1},\chi)$ in the sense of \cite[Definition 1.3]{CST14}. Since $\CR'$ and $\CR$ only differs at $3$. It suffices to verify that $\beta_3^0(f_3,f_3)\neq 0$. We delay this verification to the next subsection Proposition \ref{local}.
\end{proof}

Let $\omega_{E_n}$ be the invariant differential on the minimal model of $E_n$.  Define the minimal real period $\Omega_n$ of $E_n$ by
$$\Omega_{n}=\int_{E_n(\BR)}|\omega_{E_n}|.$$
By \cite[Formula (9)]{ZK}, we have
\begin{equation}\label{equation4}
\Omega_{3p}\Omega_{9p^2}=\Omega_{3p^2}\Omega_{9p}=p^{-1}\cdot\Omega_1^2
\end{equation}
 Using Sage we compute that $\{\Omega_{1},\Omega_{1}\cdot(\frac{1}{2}+\frac{\sqrt{-3}}{2})\}$ is a $\BZ$-basis of the period lattice $L$ of the minimal model of $E_1$. So

\begin{equation}\label{period}
\sqrt{3}\Omega^2_{1}=2\int_{\BC/L}dxdy=\int_{E(\BC)}|\omega_{E_1}\wedge\overline{\omega}_{E_1}|=8\pi^2(\phi,\phi)_{\Gamma_0(27)},
\end{equation}
where $\phi$ is the newform of level $3^3$ and weight $2$ associated to $E_1$, and  $(\phi,\phi)_{\Gamma_0(3^3)}$ is the Petersson norm of $\phi$ defined by
$$(\phi,\phi)_{\Gamma_0(27)}=\int\int_{\Gamma_0(3^3)}|\phi(z)|^2dxdy,\ \ \ \ z=x+iy.$$

Let $P_\tau=[\tau,1]$ be the CM point on $X_0(27)(H_{9p})$.  Define the Heegner points
 \[R_1=\Tr_{H_{9p}/L_{(3p)}} P_\tau\in E_1(L_{(3p)}),\]
 and
 \[R_2=\Tr_{H_{9p}/L_{(3p^2)}} P_\tau\in E_1(L_{(3p^2)}).\]
 Recall 
 \[z=\Tr_{H_{9p}/L_{(3,p)}},\]
and from (\ref{decom1}), we have the decomposition
 \[z=z_1+z_2\in E_1(L_{(3,p)})_\BQ,\]
 where $z_1$ resp. $z_2$ gives rise to a point in $E_{3p}(K)_\BQ$ resp. $E_{3p^2}(K)_\BQ$. The Heegner points $R_1$ and $R_2$ are related to $z_1$ and $z_2$ as follows.
 \begin{prop}\label{decomposition}
In $E_1(L_{(3,p)})_\BQ$, we have
\[R_1=3z_1\textrm{ and }R_2=3z_2. \]
\end{prop}
\begin{proof}
Denote
\[\sigma=\sigma_{1+3\omega_3}\textrm{ and }\tau=\sigma_{\omega_3}^i,\]
where $i=1$ if $p\equiv 2\mod 9$ and $i=2$ otherwise.
By Proposition \ref{LCF}, 
\[(\sqrt[3]{3})^{\sigma-1}=\omega^2,\quad (\sqrt[3]{p})^{\sigma-1}=1,\]
and
\[(\sqrt[3]{3})^{\tau-1}=1,\quad (\sqrt[3]{p})^{\tau-1}=\omega.\]
Then
\[\Gal(L_{(3,p)}/L_{(3p)})=\langle \sigma \tau\rangle, \quad \Gal(L_{(3,p)}/L_{(3p^2)})=\langle \sigma \tau^2\rangle.\]
Since $\phi(z_1)$ is $\chi_p$-eigen while $\phi(z_2)$ is $\chi_{p^2}$-eigen (see the proof of Theorem \ref{main}), 
\begin{eqnarray*}
R_1=\Tr_{L_{(3,p)}/L_{(3p)}}z&=&(z_1+z_1^{\sigma\tau}+z_1^{\sigma^2\tau^2})+(z_2+z_2^{\sigma\tau}+z_2^{\sigma^2\tau^2})\\
&=&3z_1+[1+\omega+\omega^2]z_2=3z_1\in E_1(L_{(3,p)})_\BQ.
\end{eqnarray*}
Similarly, we have
\[R_2=3z_2\in E_1(L_{(3,p)})_\BQ.\]

\end{proof}
Therefore, if the  Heegner point $R_1$ resp. $R_2$ is of infinite order, then $3p$ resp. $3p^2$ is a cube sum. These Heegner points satisfy the following explicit Gross-Zagier formulae.
\begin{thm} \label{thm:GZ}
Let  $p\equiv 2,5 \mod 9$ be odd prime numbers.  We  have the following explicit height formula of Heegner points:
\[\frac{L'(1,E_{3p})L(1,E_{9p^2})}{\Omega_{3p}\Omega_{9p^2}}=\wh{h}_\BQ(R_1),\]
and
\[\frac{L'(1,E_{3p^2})L(1,E_{9p})}{\Omega_{3p^2}\Omega_{9p}}=  \wh{h}_\BQ(R_2).\]

\end{thm}
\begin{proof}
Let $\CR'$ be the admissible order  for the pair $(\pi_{E_1},\chi)$ and let $f'\neq 0$ be a test vector in $V(\pi_{E_1},\chi)$ which is defined in \cite[Definition 1.4]{CST14}. The newform $f$ only differs from $f'$ at the local place $3$. Define the Heegner cycle
$$P^0_{\chi}(f)=\frac{\sharp \Pic(\CO_p)}{\Vol(\widehat{K}^\times/K^\times\widehat{\BQ}^\times,dt)}
\int_{K^\times\widehat{\BQ}^\times\backslash \widehat{K}^\times}f(P_\tau)^{\sigma_t}\chi(t)dt,$$
 and define $P^0_{\chi^{-1}}(f)$ similarly as in \cite[Theorem 1.6]{CST14}.
It follows from Proposition \ref{local} that
\[\frac{\beta_3^0(f'_3,f'_3)}{\beta_3^0(f_3,f_3)}=12.\]
By \cite[Theorem 1.6]{CST14}, We have
$$L'(1,E_1,\chi)=6\cdot\frac{(8\pi^2)\cdot(\phi,\phi)_{\Gamma_0(27)}}{\sqrt{3}p\cdot(f,f)_{\CR'^{\times}}}\cdot\left\langle P^0_{\chi}(f), P^0_{\chi^{-1}}(f)\right\rangle_{K,K},$$
where $(\cdot,\cdot)_{\CR'}$ is the pairing on $\pi_{E_1}\times \pi_{{E_1}^\vee}$ defined as in \cite[page 789]{CST14}, and $\langle\cdot,\cdot\rangle_{K,K}$ is a pairing from $E_1(\ov{K})_\BQ\times_K E_1(\ov{K})_\BQ$ to $\BC$  such that $\langle\cdot,\cdot\rangle_{K}=\rm{Tr}_{\BC/\BR}\langle\cdot,\cdot\rangle_{K,K}$ is the Neron-Tate height over the base field $K$, see \cite[page 790]{CST14}. In our case, by \cite[Lemma 2.2 and Lemma 3.5]{CST14},
$$(f,f)_{\CR'}=\frac{\Vol(X_{\CR^{' \times}})}{\Vol(X_{\CR^\times})}\deg f=\frac{\Vol(\CR^\times)}{\Vol(\CR^{'\times})}=2/3.$$
So, we get
\begin{equation}
L'(1,E_1,\chi)=9\frac{(8\pi^2)\cdot(\phi,\phi)_{\Gamma_0(3^3)}}{\sqrt{3}p}\cdot\left\langle P^0_{\chi}(f), P^0_{\chi^{-1}}(f)\right\rangle_{K,K}.
\end{equation}

On the other hand
$$P^0_{\chi}(f)=\frac{\# \Pic(\CO_p)}{\# \Pic(\CO_{9p})}\sum_{t\in\Pic(\CO_{9p})}f(P_\tau)^{\sigma_t}\chi(t).$$
Since
$$\frac{\#\Pic(\CO_p)}{\#\Pic(\CO_{9p})}=[K^\times \wh{\CO}_p^\times:K^\times \wh{\CO}_{9p}^\times]^{-1}=\frac{1}{9},$$
we have
$$P^0_{\chi}(f)=\frac{1}{9}\sum_{t\in\Pic(\CO_{9p})}f(P_\tau)^{\sigma_t}\chi(t).$$
In the following of the proof we suppose $\chi=\chi_{3p}$ and the case for $\chi=\chi_{3p^2}$ follows the same way. We have
\begin{eqnarray*}\langle P^0_{\chi}(f),P^0_{\chi^{-1}}(f)\rangle_{K,K}&=&\frac{1}{9^2}
\langle\sum_{\sigma\in\Gal(L_{(3p)}/K)}R_1^{\sigma}\chi^{-1}(\sigma),\sum_{\sigma\in\Gal(L_{(3p)}/K)}
R_1^{\sigma}\chi(\sigma)\rangle_{K,K}\\
&=&\frac{1}{27}\langle R_1,\sum_{\sigma\in\Gal(L_{(3p)}/K)}R_1^{\sigma}\chi(\sigma)\rangle_{K,K}\\
&=&\frac{1}{27}\left(\langle R_1,R_1\rangle_{K,K}-\left\langle R_1,R_1^{\sigma_{1+3\omega_3}}\right\rangle_{K,K}\right).
\end{eqnarray*}
Here we note that $\Gal(L_{(3p)}/K)=\langle \sigma_{1+3\omega_3}\rangle$.
By Corollary \ref{galois}, we have $R_1^{\sigma_{1+3\omega_3}}=[\omega^2]R_1$, then
$$\left\langle R_1,R_1^{\sigma_{1+3\omega_3}}\right\rangle_{K,K}=\frac{1}{2}\left(\widehat{h}_K([1+\omega^2]R_1)
-\widehat{h}_K([\omega^2]R_1)-\widehat{h}_K(R_1)\right).$$
Since $|1+\omega^2|=|\omega^2|=1$, by definition, $\widehat{h}_K([1+\omega^2]R_1)=\widehat{h}_K([\omega^2]R_1)=\widehat{h}_K(R_1)$. Then
$$
\left\langle R_1,R_1^{\sigma_{1+3\omega_3}}\right\rangle_{K,K}=-\frac{1}{2}\widehat{h}_K(R_1),$$
and hence
\begin{equation}\label{equation2}
\left\langle P^0_{\chi}(f),P^0_{\chi^{-1}}(f)\right\rangle_{K,K}=\frac{1}{18}\widehat{h}_K(R_1)=\frac{1}{9}\widehat{h}_\BQ(R_1).
\end{equation}

Finally, combining  (\ref{equation4})-(\ref{equation2}), we get
\[\frac{L'(1,E_{3p})L(1,E_{9p^2})}{\Omega_{3p}\Omega_{9p^2}}=\wh{h}_\BQ(R_1).\]
\end{proof}

\subsection{Local matrix coefficients}
Define the local matrix coefficient
\[\Phi(t)=\frac{(\pi_3(t)f_3,f_3)}{(f_3,f_3)},\quad t\in \GL_2(\BQ_3).\]
Since $\chi$ has conductor $9p$, we have
\begin{eqnarray}\label{beta}
\beta^0_3(f_3,f_3)&=&\frac{\Vol(K_3^\times/\BQ_3^\times)}
{ |K_3^\times/\BQ_3^\times\CO_{9p,3}^\times|}\sum_{t\in K_3^\times/\BQ_3^\times\CO_{9p,3}^\times}\Phi(t)\chi_{3}(t)\\
&=&\frac{\Vol(K_3^\times/\BQ_3^\times)}
{ 18}\sum_{t\in S\bigsqcup S'}\Phi(t)\chi_{3}(t),\nonumber
\end{eqnarray}
where
\[S=\{1+y\sqrt{-3}| y\in \BZ/9\BZ\},\quad S'=\{3y+\sqrt{-3}| y\in \BZ/9\BZ\}.\]
Note that $S\bigsqcup S'$ is a complete system of representatives of $K_3^\times/\BQ_3^\times\CO_{9p,3}^\times$. In the above, we view $\chi$ as an adelic character through the class field theory and $\chi_3$ is the $3$-part of $\chi$. In order to compute $\beta_3^0(f_3,f_3)$, it suffices to compute the local matrix coefficients $\Phi(t), t\in S\bigsqcup S'$.

For the basic knowledge and conventions from representation theory we need to compute these matrix coefficients, we refer to \cite[\S 4.2.2]{SY17}. Let $\psi$ be the additive character such that $\psi(x)=e^{2\pi i \iota(x)}$ where $\iota:\BQ_3\rightarrow \BQ_3/\BZ_3 \subset \BQ/\BZ$ is the map given by $x\mapsto -x\mod \BZ_3$ and put $\psi^-(x)=\psi(-x)$. Let $dx$ be the Haar measure on $\BQ_3$ which is self-dual with respect to $\psi$, and we fix a Haar measure $d^\times x$ on $\BQ_3^\times$ such that $\Vol(\BZ_3^\times)=1$. The local representation $\pi_3=\pi_{E_1,3}$ is supercuspidal, and the Kirillov model $\sK(\pi_3,\psi)$ is the unique realization of $\pi_3$ on the Schwartz function space $\CS(\BQ_3^\times)$ such that
$$\pi\left(\matrixx{a}{b}{0}{1}\right)\phi(x)=\psi( bx)\phi(ax),\quad \phi\in \CS(\BQ_3^\times).$$
The $\GL_2(\BQ_3)$-invariant pairing on $\pi_3\times \pi_3$ is given by
\[(\phi_1,\phi_2)=\int_{\BQ_3^\times}\phi_1(x)\ov{\phi_2(x)}d^\times x.\]
Put
$$1_{\nu,n}(x)=\begin{cases}\nu(u),& \textrm{if $ x=u3^{n}$ for $ u\in\BZ_3^{\times}$},\\
                             0,& \mathrm{otherwise},\end{cases}$$
where $\nu$ is a character of $\BZ_3^{\times}$.
Then $\{1_{\nu,n}(x)\}_{\nu,n}$ is an orthogonal basis of $\CS(\BQ_3^\times)$ with respect to the paring $(\cdot,\cdot)$. For $\phi(x)\in\CS(\BQ_3^\times)$, we have the Fourier expansion
$$\phi(x)=\sum_{\nu}\sum_n\widehat{\phi}_n(\nu^{-1})1_{\nu,n},$$
where
$$\widehat{\phi}_n(\nu^{-1})=\int_{\BZ_3^\times}\phi(3^nx)\nu^{-1}(x)d^\times x.$$
The action of $w=\matrixx{0}{1}{-1}{0}$ on $1_{\nu,n}$ can be described as follows:
$$\pi_3\left(\matrixx{0}{1}{-1}{0}\right)1_{\nu,n}=C_\nu 1_{\nu^{-1},-n+n_\nu}.$$
where
$$C_{\nu}=\epsilon(1/2,\pi\otimes \nu^{-1}, \psi),$$
and $n_{\nu}=-\max\{5, 2i\}$, where $i$ is the conductor of $\nu$. From $w^2=-1$, one can see that
$$n_{\nu}=n_{\nu^{-1}},\ \ \ C_{\nu}C_{\nu^{-1}}=1.$$
It is well known that $1_{1,0}$ is the normalized local  new form, and hence, is  parallel to $f_3$.

Note that $\Phi(t)$ is bi-invariant under $U_0(3^3)_3$ and hence induces a function on the double cosets
$ZU_0(3^3)_3\backslash \GL_2(\BQ_3)/U_0(3^3)_3$, where $Z$ denotes the center of $\GL_2(\BQ_3)$. Consider
the natural projection map
$$\pr:\GL_2(\BQ_3)\rightarrow ZU_0(3^3)_3\backslash \GL_2(\BQ_3)/U_0(3^3)_3.$$

Under the embedding $\rho: K\hookrightarrow \M_2(\BQ)$, we write
\begin{equation}\label{matrix}
\sqrt{-3}=\matrixx{-1}{-2p/9}{18/p}{1}.
\end{equation}
We have
the following decomposition of matrix:
\begin{equation}\label{decomposition1}
\matrixx{a}{3^ib}{3^jc}{d}=\matrixx{(ac^{-1}-bd^{-1}3^{j+i})}{d^{-1}b3^i}{0}{1}\matrixx{1}{0}{3^j}{1}\matrixx{c}{0}{0}{d}.
\end{equation}

Let $i<3$ be an integer. From
$$\matrixx{1}{0}{3^i}{1}=-w\matrixx{1}{-3^i}{0}{1}w,$$
we have
$$\pi_3\left(\matrixx{1}{-3^i}{0}{1}w\right)1_{1,0}(x)=C_1\psi^{-}(3^ix)1_{1,-3}(x)=C_1\sum_{c(\nu)=3-i}\widehat{\psi^-}_{i-3}(\nu^{-1})1_{\nu,-3}(x),$$
\[\pi_3\left(\matrixx{1}{0}{3^i}{1}\right)1_{1,0}=\sum_{c(\nu)=3-i}C_1C_{\nu}\widehat{\psi^-}_{i-3}(\nu^{-1})1_{\nu^{-1},\min(0,2i-3)},\]
that is

\begin{equation}\label{rep6}
\pi_3\left(\matrixx{1}{0}{3^i}{1}\right)1_{1,0}=\begin{cases}
\sum_{c(\nu)=3-i}C_1C_{\nu}\widehat{\psi^-}_{i-3}(\nu^{-1})1_{\nu^{-1},2i-3},\quad & \textrm{if $ i\leq
1$};\\
\wh{\psi^-}_{-1}(1)1_{1,0}+C_1C_{\nu}\widehat{\psi^-}_{-1}(\nu^{-1})1_{\nu^{-1},0},\quad & \textrm{if $i=2$}.
\end{cases}
\end{equation}
Under the embedding $\rho$, we have
$$3y+\sqrt{-3}=\matrixx{3y-1}{3^{-2}(-2p)}{3^2(2/p)}{3y+1},\quad y\in \BZ/9\BZ$$
We have the decomposition
\begin{equation*}
\matrixx{3y-1}{3^{-2}(-2p)}{3^2(2/p)}{3y+1}=\matrixx{\frac{3(3y^2+1)p}{2(3y+1)}}{\frac{-2p}{3^2(3y+1)}}{0}{1}
\matrixx{1}{0}{3^2}{1}\matrixx{2/p}{0}{0}{3y+1}.
\end{equation*}
Then
\begin{eqnarray*}
\pr\left(\matrixx{3y-1}{3^{-2}(-2p)}{3^2(2/p)}{3y+1}\right)&=&
\matrixx{\frac{3(3y^2+1)p}{2(3y+1)}}{\frac{-2p}{3^2(3y+1)}}{0}{1}\matrixx{1}{0}{3^{2}}{1}\\
&=&\matrixx{3A_y}{\frac{B_y}{3^{2}}}{0}{1}\matrixx{1}{0}{3^{2}}{1}
\end{eqnarray*}
where $A_y=\frac{(3y^2+1)p}{2(3y+1)}$ and
$B_y=\frac{-2p}{3y+1}$ are 3-units.
Then by (\ref{rep6})
\begin{eqnarray*}\label{rep1}
\pi_3(3y+\sqrt{-3})1_{1,0}(x)&=&\psi\left(\frac{B_yx}{3^{2}}\right)
\widehat{\psi^-}_{-1}(\nu_1)C_{\nu_1}1_{\nu_1,0}\left(3A_yx\right)+\psi\left(\frac{B_yx}{3^{2}}\right)
\widehat{\psi^-}_{-1}(1)1_{1,0}\left(3A_yx\right)\\
&=&\psi\left(\frac{B_yx}{3^{2}}\right)\left(\frac{\sqrt{-3}}{2}
\nu_1(A_y)C_{\nu_1}1_{\nu_1,-1}\left(x\right)+\frac{-1}{2}
1_{1,-1}\left(x\right)\right),
\nonumber
\end{eqnarray*}
and hence we see                                                                                                                                                       \begin{equation}\label{eqn:mc1}
\Phi(3y+\sqrt{-3})=0,\quad y\in \BZ/9\BZ.
\end{equation}
For $y\in \BZ/9\BZ$, we have
\begin{eqnarray}\label{eqn:mc2}
\Phi(1+y\sqrt{-3})&=&(\pi_3(1+y\sqrt{-3})1_{1,0},1_{1,0})\\
&=&(\pi_3(-3y+\sqrt{-3})1_{1.0},\pi_3(\sqrt{-3})1_{1,0})))\nonumber\\
&=&\int_{v(x)=-1}\psi\left(\frac{yx}{-3(3y-1)}\right)d^\times x\nonumber\\
&=&\begin{cases}1,&\quad y=0;\\ -\frac{1}{2}, &\quad y=3,\ 6;\\ 0, &\quad\text{otherwise}.\end{cases}\nonumber
\end{eqnarray}


\begin{prop}\label{local}
The local integral
\[\beta^0_3(f_3,f_3)=\frac{\Vol(K_3^\times/\BQ_3^\times)}{12}.\]
\end{prop}
\begin{proof}
It follows from (\ref{beta}), (\ref{eqn:mc1}) and (\ref{eqn:mc2}) that
\begin{eqnarray*}
\beta^0_3(f_3,f_3)
&=&\frac{\Vol(K_3^\times/\BQ_3^\times)}
{ 18}\sum_{t\in S\bigsqcup S'}\Phi(t)\chi_{3}(t)\\
&=&\frac{\Vol(K_3^\times/\BQ_3^\times)}
{ 18}\left(1-\frac{1}{2}(\omega+\omega^2)\right)\\
&=&\frac{\Vol(K_3^\times/\BQ_3^\times)}{12}.
\end{eqnarray*}

\end{proof}

\subsection{Variant of the Birch and Swinnerton-Dyer conjecture}
By the explicit descent method of $3$-isogenies, we can compute the following Selmer groups.
\begin{prop}\label{selmer}
If $p\equiv 2,5\mod 9$ is an odd prime, then \[\dim_{\BF_3} \Sel_3(E_{3p})=1\textrm{ and }\dim_{\BF_3} \Sel_3(E_{9p^2})=0,\]
and 
\[\dim_{\BF_3} \Sel_3(E_{3p^2})=1\textrm{ and }\dim_{\BF_3} \Sel_3(E_{9p})=0.\]
\end{prop}
\begin{proof}
The elliptic curve $E_{3p}$ has Weierstrass equation $y^2=x^3-2^43^3(3p)^2$ and let $E_{3p}'$ be the elliptic curve given by Weierstrass equation $y^2=x^3+2^4(3p)^2$. There is a unique isogeny $\phi:E_{3p}\ra E_{3p}'$ of degree $3$ up to automorphisms $\pm 1$. Let $\phi':E_{3p}'\ra E_{3p}$ be the dual isogeny of $\phi$. By \cite[Theorem 2.9]{Satge},
\[\Sel_\phi(E_{3p})\simeq (\BZ/3\BZ)^2\textrm{ and }\Sel_{\phi'}(E_{3p}')=0.\]
Note $E_{3p}(\BQ)_\tor=0$, $E'_{3p}(\BQ)_\tor\simeq \BZ/3\BZ$ and $\Sha(E'_{3p})[\phi']=0$. By \cite[Lemma 5.1]{SY17}, we have
\[\dim_{\BF_3}\Sel_3(E_{3p})=1.\]
The Selmer groups $\Sel_3(E_{3p^2})$, $\Sel_3(E_{9p})$ and $\Sel_3(E_{9p^2})$ can be computed the same way.
\end{proof}
By  the formula for epsilon factors in \cite{Liverance}, we know
\[\epsilon(E_{3p})=\epsilon(E_{3p^2})=-1\textrm{ and }\epsilon(E_{9p})=\epsilon(E_{9p^2})=+1.\]
Then the Birch and Swinnerton-Dyer (B-SD) conjecture would imply that 
\[\rk E_{3p}(\BQ)=\rk E_{3p^2}(\BQ)=1 \textrm{ and }\rk E_{9p}(\BQ)=\rk E_{9p^2}(\BQ)=0.\]

Suppose the Heegner point $R_1$ is not torsion. It follows from the work of Kolyvagin \cite{Kolyvagin1990} and Gross-Zagier \cite{GZ1986} that
\[\rk_\BZ E_{3p}(\BQ)=\ord_{s=1}L(s,E_{3p})=1,\quad \rk_\BZ E_{9p^2}(\BQ)=\ord_{s=1}L(s,E_{9p^2})=0,\]
and the Shafarevich-Tate groups $\Sha(E_{3p})$ and $\Sha(E_{9p^2})$ are finite.

Let $P_1$  be a generator of the free part of $E_{3p}(\BQ)$. By \cite[Table 1]{ZK}, we know that the Tamagawa numbers $c_\ell(E_{3p})=c_{\ell}(E_{9p^2})=1$ for all primes $\ell$. Then the B-SD conjecture predicts that
\begin{equation*}
\frac{L'(1,E_{3p})}{\Omega_{3p}}=|\Sha(E_{3p})|\cdot\widehat{h}_\BQ(P_1),
\end{equation*}
 and
\begin{equation*}
\frac{L(1,E_{9p^2})}{\Omega_{9p^2}}=\left|\Sha(E_{9p^2})\right|.
\end{equation*}
Combining these two, we get
\begin{equation}\label{bsd}
\frac{L'(1,E_{3p})}{\Omega_{3p}}\cdot\frac{L(1,E_{9p^2})}{\Omega_{9p^2}}=|\Sha(E_{3p})|\cdot|\Sha(E_{9p^2})|\cdot\widehat{h}_\BQ(P_1).
\end{equation}
Comparing this with the explicit height formula Theorem \ref{thm:GZ}, we have
\begin{equation}\label{bsd1}
|\Sha(E_{3p})|\cdot|\Sha(E_{9p^2})|=\frac{\wh{h}_\BQ(R_1)}{\widehat{h}_\BQ(P_1)}.
\end{equation}

Suppose both the Heegner point $R_2$ is not torsion.  Let $P_2$ be a generator of the free part of $E_{3p^2}(\BQ)$.  The B-SD conjecture implies via a similar argument  that
\begin{equation}\label{bsd2}
|\Sha(E_{3p^2})|\cdot|\Sha(E_{9p})|=\frac{\wh{h}_\BQ(R_2)}{\widehat{h}_\BQ(P_2)}.
\end{equation}
It is easy to see that the rations $\frac{\wh{h}_\BQ(R_1)}{\widehat{h}_\BQ(P_1)}$ and $\frac{\wh{h}_\BQ(R_2)}{\widehat{h}_\BQ(P_2)}$ are rational numbers.  By the work of Perrin \cite{PR1987} and Kobayashi \cite{Koba2013}, for good primes $\ell$, i.e. $\ell\nmid 6p$, the $\ell$-part of the identities (\ref{bsd1}) and (\ref{bsd2})  hold. 

Note by Proposition \ref{selmer}, we have
$$|\Sha(E_{3p})[3^\infty]|=|\Sha(E_{9p^2})[3^\infty]|=1.$$ 
Then we have the following variant of the $3$-part of the B-SD conjectures for the related elliptic curves.
\begin{prop}
Suppose the Heegner point $R_1$ resp. $R_2$ is not torsion.   Let $P_1$ resp. $P_2$ be a generator of the free part of $E_{3p}(\BQ)$ resp.  $E_{3p^2}(\BQ)$. Then the $3$-part of $\Sha(E_{3p})|\cdot|\Sha(E_{9p^2})| $resp. $|\Sha(E_{3p^2})|\cdot|\Sha(E_{9p})|$ is as predicted by the B-SD conjecture if and only if 
\[\frac{\wh{h}_\BQ(R_1)}{\widehat{h}_\BQ(P_1)}\textrm{ resp. } \frac{\wh{h}_\BQ(R_2)}{\widehat{h}_\BQ(P_2)}\]
is a  $3$-adic unit.
\end{prop}

\bibliographystyle{alpha}
\bibliography{reference}
\end{document}